\documentclass[reqno]{amsart}
\usepackage{amsmath,amsbsy,amsfonts,enumerate}
\usepackage{color}
\date{}
\def\dis{\displaystyle}
\def\nd{\noindent}
\def\thend{\rule{3mm}{3mm}}
\def\Re{\mathbb{R}}
\newtheorem{theorem}{Theorem}[section]
\newtheorem{cor}{Corollary}[section]

\newtheorem{lem}{Lemma}[section]
\newtheorem{rmk}{Remark}[section]

\begin{document}

\title[Ground states of elliptic problems]{Ground states of elliptic problems involving non homogeneous operators}
\vspace{1cm}

\author{Giovany M. Figueiredo}
\address{Giovany M. Figueiredo \newline  Universidade Federal do Par\'a, Faculdade de Matem\'atica, 66075-110, Bel\'em-PA, Brazil}
\email{\tt giovany@ufpa.br} 

\author{Humberto Ramos Quoirin}
\address{H. Ramos Quoirin \newline Universidad de Santiago de Chile, Casilla 307, Correo 2, Santiago, Chile}
\email{\tt humberto.ramos@usach.cl}

\thanks{The first author was supported by CNPq/PQ
301242/2011-9. The second author was supported by
 FONDECYT 11121567}

\begin{abstract}
We investigate the existence of ground states for  functionals with nonhomogenous principal part. Roughly speaking, we show that the Nehari manifold method requires no homogeinity on the principal part of a functional. This result is motivated by some elliptic problems involving nonhomogeneous operators. As an application, we prove the existence of a  ground state and infinitely many solutions for three classes of boundary value problems.
\end{abstract}

\maketitle

\section{Introduction and main results}

This article is concerned with a class of variational elliptic problems involving non-homogeneous operators. Our main goal is to provide a unified approach to obtain {\it ground state} solutions for these problems. This approach is based on the Nehari manifold method, which was introduced in \cite{N} and is by now a well-established and useful tool in finding solutions of problems with a variational structure, cf. \cite{Ad,AC,BB,BWW,BDH,BZ,BW,CCN,CZ,Ta}. In an abstract setting, given a Banach space $X$ and a $\mathcal{C}^1$ functional $\Phi:X \rightarrow \Re$, 
a {\it ground state} of $\Phi$ is a solution $u_0$ of the problem 
$$\Phi'(u_0)=0, \quad \Phi(u_0)=\min\{ \Phi(u); \ u \text{ is a critical point of } \Phi \}.$$
When looking for such a solution, one may restrict $\Phi$ to the set
$$\mathcal{N}=\{u \in X \setminus\{0\};\ \Phi'(u)u=0\},$$ which not only contains all nontrivial critical points of $\Phi$, but also turns out to have some useful properties. It is well-known that under some conditions on $\Phi$, $\mathcal{N}$ is a $\mathcal{C}^1$ submanifold of $X$ and critical points of the restriction of $\Phi$ to $\mathcal{N}$ are in fact critical points of $\Phi$. As an immediate consequence, one may obtain a {\it ground state} of $\Phi$ by minimizing $\Phi$ over $\mathcal{N}$.
An overview of this procedure, as well as further developments and several applications of this method, are given in \cite{SW}.

Throughout this article we assume that $\Omega \subset \Re^N$ ($N \geq 3$) is a bounded domain and $f:\Re \rightarrow
[0,\infty)$ is an odd $\mathcal{C}^{1}$ function, so that $F(t)=\displaystyle\int_{0}^{t}f(s)\ ds$, defined for $t \in \Re$, is an even function. A typical application of the Nehari manifold method  provides the existence of a ground state for the prototype problem
\begin{equation}
\label{p1}
-\Delta u=f(u), \quad u \in H_0^1(\Omega), \quad u \geq 0, \end{equation}
where $\Delta$ is the Laplace operator and $f$ is, in addition, subcritical and superlinear. More precisely, $f$ satisfies
$$|f(s)|\leq C(1+|s|^{r-1}) \quad \forall s \in \Re$$
for some $C>0$ and $r \in (2,2^*)$, where $2^*=\frac{2N}{N-2}$, as well as \begin{equation}
\label{hf}
\lim_{s\to 0}\frac{f(s)}{s}=0, \quad \lim_{s \to \infty} \frac{f(s)}{s}=\infty, \quad \text{
and } \quad \frac{f(s)}{s} \text{ is increasing in } (0,\infty).
\end{equation}
Then the functional $$\Phi(u):= \frac{1}{2}\|u\|^2- \int_\Omega F(u),$$
defined on $H_0^1(\Omega)$,
has a non-negative and nontrivial {\it ground state} $u_0$, which is a classical solution of \eqref{p1}.
This result has a natural extension to an abstract setting as follows:

Let $X$ be a uniformly convex Banach space and $\mathcal{S}$ be the unit sphere in $X$. Assume that $\| \cdot \|$ is a $\mathcal{C}^1$ functional on $X \setminus \{0\}$. 
Then the following result holds, cf. \cite[Theorem 13]{SW}:

\begin{theorem}\cite{SW}\label{tsw}
Let $\Phi$ be such that $\Phi(0)=0$ and $\Phi=I_0-I$ where $I_0,I$ are $\mathcal{C}^1$ functionals on $X$ satisfying, for some $p>1$:
\begin{enumerate} 
\item $I'(u)=o(\|u\|^{p-1})$ as $u \to 0$.
\item $s \mapsto \frac{I'(su)}{s^{p-1}}$ is strictly increasing in $(0,\infty)$ for every $u \neq 0$.
\item $\frac{I(su)}{s^p} \to \infty$ uniformly for $u$ on weakly compact subsets of $X \setminus \{0\}$ as $s \to \infty$.
\item $I'$ is completely continuous. i.e. if $u_n \rightharpoonup u$ in $X$ then $I'(u_n) \rightarrow I'(u)$ in $X'$.
\item $I_0$ is weakly lower semicontinuous, positively homogeneous of degree $p$, i.e. $I_0(su)=s^pI_0(u)$, and satisfies
$$C_0\|u\|^p\leq I_0(u)\leq C_0^{-1}\|u\|^p$$ 
and 
\begin{equation}
\label{ip} \left(I_0'(v)-I_0'(w)\right)(v-w)\geq C_1\left( \|v\|^{p-1} - \|w\|^{p-1}\right)(\|v\|-\|w\|)
\end{equation} for some $C_0,C_1>0$ and every $u,v \in X$.
\end{enumerate}
Then the equation $\Phi'(u)=0$ has a ground state solution. Moreover, if $\Phi$ is even then this equation has infinitely many pairs of solutions.
\end{theorem}

The above theorem is clearly motivated by problems involving the $p$-Laplace operator, namely, 
\begin{equation}
\label{p2}-\Delta_p u =f(u), \quad u \in W_0^{1,p}(\Omega),
\end{equation}
where $\Delta_p u =div \left(|\nabla u|^{p-2} \nabla u\right)$ and $p>1$. In this case $X=W_0^{1,p}(\Omega)$ with $\|u\|=\left(\int_\Omega |\nabla u|^p\right)^{\frac{1}{p}}$ and $\Phi=I_0-I$, where
$$I_0(u)=\frac{1}{p} \|u\|^p\quad
\text{and} \quad I(u)=\int_\Omega F(u).$$
Under  conditions similar to \eqref{hf}, it can be shown that $I$ satisfies the assumptions of Theorem \ref{tsw}, so that \eqref{p2} has a ground state solution and infinitely many pairs of solutions.

The following geometrical properties of $\Phi$ are essential in the proof of Theorem \ref{tsw}:
\begin{itemize}
\item [(A2)] For any $w \in X \setminus \{0\}$ the map $t \mapsto \Phi(tw)$, defined for $t>0$, has a unique critical point $t_w>0$ which satisfies $\dis \Phi(t_w w)=\max_{t>0} \Phi(tw)$.\\
\item [(A3)] $t_w$ is uniformly bounded away from zero for $w \in \mathcal{S}$, i.e. there exists $\delta>0$ such that $t_w\geq \delta$ for every $w \in \mathcal{S}$. Moreover, $t_w$ is bounded from above for $w$ in a compact subset of $\mathcal{S}$, i.e.
given a compact set $\mathcal{W} \subset \mathcal{S}$ there exists $C_{\mathcal{W}}>0$ such that $t_w \leq C_{\mathcal{W}}$ for every $w \in \mathcal{W}$.
\end{itemize}
\medskip

These properties are used to show that $\mathcal{S}$ is homeomorphic to $\mathcal{N}$ through 
the projection $w \mapsto t_w w$ and that one may carry out a critical point theory on $\mathcal{N}$, cf. \cite[Corollary 10]{SW}. We shall prove that (A2) and (A3) hold for a larger class of functionals, in particular, for  $\Phi=I_0-I$, where $I_0$ is not positively homogeneous and $I(u)=\int_\Omega F(u)$. This situation is motivated by the following examples:\\
\begin{enumerate}
\item $X=W_0^{1,p}(\Omega)$ and 
$$I_0(u)=\frac{1}{p}\int_{\Omega}A(|\nabla u|^p).$$
Here $A(s)=\int_0^s a(t) dt$ and $a:[0,\infty)\rightarrow
[0,\infty)$ is $\mathcal{C}^{1}$ in $(0,\infty)$ and satisfies
$$k_0\left( 1+t^{\frac{q-p}{p}}\right) \leq a(t) \leq k_1\left( 1+t^{\frac{q-p}{p}}\right) \quad \forall t>0,$$ where $k_0,k_1$ are positive constants and $p\geq q>1$. The associated Euler-Lagrange equation is the quasilinear equation
\begin{equation}\label{quasi}
-div \left(a(|\nabla u|^p)|\nabla u|^{p-2}\nabla u\right) = f(u), \quad u \in W_0^{1,p}(\Omega).
\end{equation}
This class of operators contains the $p$-Laplacian ($a(t) \equiv 1$),  as well as the sum of the $p$-Laplacian and the $q$-Laplacian ($a(t)=1+t^{\frac{q-p}{p}}$). Problems involving this class of operators have been investigated, for instance, in \cite{BMV,CI,CD,Fi2, HP, MP}.\\

\item $X=H_0^1(\Omega)$ and $$I_0(u)=\frac{1}{2}\hat{M}\left(\|u\|^2\right),$$ where $\|u\|=\left(\int_\Omega |\nabla u|^2\right)^{\frac{1}{2}}$, $\hat{M}(s)=\int_0^s M(t)dt$ and $M:[0,\infty)\rightarrow [0,\infty)$ is a continuous function. In this case, the corresponding Euler-Lagrange equation is the Kirchhoff type equation
\begin{equation}
\label{kir}
-M\left(\int_{\Omega} |\nabla u|^2\right) \Delta u =f(u), \quad u \in H_0^1(\Omega),
\end{equation}
which has been intensively investigated over the last years, specially for $M(t)=at+b$, with $a,b>0$, cf. \cite{ACF,Ba,Fi,HZ,Ma,Na}. 

We shall prove that this equation has a ground state for a larger class of $M$, which includes, for instance, $$M(t) = m_{0} + \ln(1 + t)$$ or $$
\displaystyle
M(t)=m_{0}+\displaystyle\sum_{i=1}^{k}b_{i}t^{\gamma_{i}}
$$
where $b_{i}\geq 0$ and $\gamma_{i}\in (0,1]$ for $i=1,...,k$, with $b_i>0$ for at least one $i$.\\

\item Let $\overrightarrow{p}=(p_1,p_2,...,p_n)$ with $p_i>1$ for $i=1,...,N$ and $\sum_{i=1}^N \frac{1}{p_i}>1$. Let
 $X=\mathcal{D}_0^{1,\overrightarrow{p}}(\Omega)$ be the completion of $\mathcal{C}_0^{\infty}(\Omega)$ with respect to the norm $\|u\|=\displaystyle \sum_{i=1}^N \|\partial_i u\|_{p_i}$.
We set, for $u \in X$, 
$$I_0(u)=\sum_{i=1}^N \frac{1}{p_i}\int_\Omega |\partial_i u|^{p_i}.$$
The corresponding Euler-Lagrange equation is the anisotropic equation
\begin{equation}
\label{anis}
-\sum_{i=1}^N \partial_i \left(|\partial_i u|^{p_i-2} \partial_i u\right) = f(u), \quad u \in \mathcal{D}_0^{1,\overrightarrow{p}}(\Omega).
\end{equation}
For results on this class of problems, we refer to \cite{Alvesel,Agnesi2,FGK, GP, Vetois1} and references therein.
\end{enumerate}
\medskip

We shall establish an abstract result (in the same style as Theorem \ref{tsw}) which applies to the problems above. In this sense, we shall prove that the Nehari manifold method applies to problems with nonhomogenous operators. To prove that $\Phi$ has a ground state, we follow a strategy slightly different from \cite{SW}, since we do not prove that $\Phi$ satisfies the Palais-Smale condition at the ground state level. In doing so, we also get rid of the condition \eqref{ip} and the uniform convexity assumption on $X$. This approach, in a rather simple setting, can be found in \cite{BS}. Once we have proved that the infimum of $\Phi$ over $\mathcal{N}$ is achieved, we shall deduce that it is a critical value of $\Phi$ thanks to the results of \cite{SW}, which apply to $\mathcal{C}^1$ functionals.

Finally, let us recall (as pointed out in \cite{SW}) that the Nehari manifold method also has the advantage of not requiring an Ambrosetti-Rabinowitz type condition on $f$, which is customary when dealing with the Palais-Smale condition.

We state now our main result:

\begin{theorem}
\label{tp}
Let $X$ be a reflexive Banach space such that $\| \cdot \|$ is a $\mathcal{C}^1$ functional on $X \setminus \{0\}$ and $\Phi: X \rightarrow \Re$ be a $\mathcal{C}^1$ functional such that $\Phi(0)=0$. In addition, we assume that there exist $p, r >1$ such that:
\begin{enumerate}
\item  $\displaystyle \liminf_{u \to 0} \frac{\Phi'(u)u}{\|u\|^r}>0$ 
\item For every $u \in X$ we have $\Phi(u)\geq C_0\|u\|^r -I(u)$
where $C_0>0$ and $I$ is a weakly continuous functional on $X$.
\item $\displaystyle \lim_{t \to \infty} \frac{\Phi(tu)}{t^p}=-\infty$ uniformly for $u$ on weakly compact subsets of $X \setminus \{0\}$.
\item For every $u \in X \setminus \{0\}$ the maps $\dis t \mapsto \frac{\Phi'(tu)u}{t^{p-1}}$ and $ t \mapsto \Phi(tu)-\frac{1}{p}\Phi'(tu)tu$ are  decreasing and increasing in $(0,\infty)$, respectively. 
\item $u \mapsto \Phi'(u)u$ and $u \mapsto \Phi(u)-\frac{1}{p}\Phi'(u)u$ are weakly lower semicontinuous on $X$.
\end{enumerate}
Then $c:=\inf_{\mathcal{N}} \Phi$ is positive and achieved by some  $u_0 \neq 0$, i.e. $\Phi$ has a nontrivial ground state at a positive level. If, in addition, $\Phi$ is even then we may choose $u_0 \geq 0$.
\end{theorem}

\begin{proof}
The proof is divided in two steps: first we show that $c$ is achieved, and then we use the results of \cite{SW} to prove that $c$ is a critical value of $\Phi$.\\

{\bf Step 1:} $c$ is achieved\\

Given $u \in X \setminus \{0\}$, we set $\gamma_u(t)=\Phi(tu)$ for $t>0$. From (1) and (3) it is clear that $\gamma_u(t)>0$
for $t$ sufficiently small and $\gamma_u(t)<0$ for $t$ sufficiently large. Consequently $\gamma_u$ has a global maximum point $t_u>0$, which is a critical point of $\gamma_u$. 
Since $t^{1-p}\gamma_u'(t)=t^{1-p}\Phi'(tu)u$, from (4) we infer that $t \mapsto t^{1-p}\gamma_u'(t)$ is decreasing. It follows that $\gamma_u'$ vanishes exactly once, i.e. $t_u$ is the unique critical point of $\gamma_u$. In particular, there holds $\Phi(u)>0$ for every $u \in \mathcal{N}$, so that $c \geq 0$.

We claim that $\mathcal{N}$ is bounded away from zero. Indeed, if $(u_n) \subset \mathcal{N}$ with $u_n \rightarrow 0$ in $X$ then $\frac{\Phi'(u_n)u_n}{\|u_n\|^r} =0$ for every $n$, which contradicts (1). Thus the claim is proved.
 
Let us prove now that if $(u_n) \subset \mathcal{N}$ is such that $(\Phi(u_n))$ is bounded from above then $(u_n)$ is bounded and, up to a subsequence, $u_n \rightharpoonup u_0$ with $u_0 \not \equiv 0$. Assume by contradiction that $(u_n) \subset \mathcal{N}$ is unbounded. Then we may assume that $\|u_n\| \to \infty$ and $v_n \rightharpoonup v_0$, where $v_n=\frac{u_n}{\|u_n\|}$. If $v_0 \equiv 0$ then, since $t=1$ is the global maximum point of $\gamma_{u_n}$, we have, using (2),
\begin{equation}
\label{b1}\Phi(u_n)\geq \Phi(tv_n)\geq C_0t^r-I(tv_n) \rightarrow C_0t^r-I(0), \quad \forall t>0.
\end{equation}
This contradicts the fact that $(\Phi(u_n))$ is bounded from above. Hence $v_0 \not \equiv 0$ and consequently, by (3),
 $$\frac{\Phi(u_n)}{\|u_n\|^p}= \frac{\Phi(\|u_n\|v_n)}{\|u_n\|^p} \rightarrow -\infty,$$
 which contradicts the fact that $\Phi(u_n)>0$ for every $n$.
 Therefore $(u_n)$ must be bounded and, up to a subsequence,  $u_n \rightharpoonup u_0$. If $u_0 \equiv 0$ then, repeating the argument used in the case $v_0 \equiv 0$, we get
 $$\Phi(u_n)\geq \Phi(tu_n)\geq C_0t^r\|u_n\|^r-I(tu_n) \geq D_0t^r -I(tu_n) \rightarrow D_0t^r-I(0),$$ where we used (2) and the fact that $\mathcal{N}$ is bounded away from zero. So we get another contradiction, which shows that $u_0 \not \equiv 0$. 
In particular, if $(u_n)$ is a minimizing sequence for $c$ then we may assume that $u_n \rightharpoonup u_0$ with $u_0 \not \equiv 0$. Let $t_0=t_{u_0}$, i.e. $t_0u_0 \in \mathcal{N}$. From (5), we infer that
$$\Phi'(u_0)u_0 \leq \liminf \Phi'(u_n)u_n=0,$$ and, as a consequence, $t_0\leq 1$. We claim that $t_0=1$. Indeed, if $t_0<1$ then, using (4) and (5), we get
\begin{eqnarray*}c&\leq& \Phi(t_0u_0)=\Phi(t_0u_0) -\frac{1}{p} \Phi'(t_0u_0)t_0u_0<\Phi(u_0) -\frac{1}{p} \Phi'(u_0)u_0\\ &\leq& \liminf \left( \Phi(u_n) -\frac{1}{p} \Phi'(u_n)u_n \right) =\lim \Phi(u_n)
=c,
\end{eqnarray*}
which is a contradiction. Therefore $t_0=1$, $u_0 \in \mathcal{N}$, and $\Phi(u_0)=c$.  Finally, if $\Phi$ is even then $\Phi(u_0)=\Phi(|u_0|)$, so that $|u_0|$ achieves $c$.\\

{\bf Step 2:} $c$ is a critical value of $\Phi$\\

From the previous step it is clear that $\Phi$ satisfies (A2) and (A3) from \cite{SW}. By \cite[Corollary 10]{SW}, we deduce that $c=\dis \inf_{\mathcal{S}} \Psi$, where $\Psi$ is defined by $$\Psi(w)=\Phi(t_w w) \quad \text{for } w \in \mathcal{S}.$$ Moreover $\Psi$ is a $\mathcal{C}^1$ functional on $\mathcal{S}$, which is a $\mathcal{C}^1$ submanifold of $X$, and $w$ is a critical point of $\Psi$ if and only if $t_w w$ is a critical point of $\Phi$. This proves that $c$ is a critical value of $\Phi$. 
\end{proof}

\begin{rmk} \label{r1}
\strut {\rm
\begin{enumerate}
\item We may easily check that the proof of Theorem \ref{tp} still can be carried out if instead of (2), the following conditions hold:\\
\begin{enumerate}
\item [(A)] For every $u\in X$ we have $\Phi(u)= I_0(u)-I(u)$,
where $I$ is weakly continuous and $I_0$ is such that $\dis \lim_{t \to \infty} I_0(tu)=\infty$ uniformly for $u \in \mathcal{S}$.\\
\item [(B)] For every $u\in X$ we have $\Phi'(u)u=J_0(u)-J(u)$, where $J$ is  weakly continuous and $J_0$ is such that $J_0(u_n) \rightarrow 0$ if and only if $u_n \to 0$.\\
\end{enumerate}
As a matter of fact, one may repeat \eqref{b1} and use (A) to get a contradiction if $(u_n) \subset \mathcal{N}$ is such that $(\Phi(u_n))$ is bounded from above and $v_n=\frac{u_n}{\|u_n\|} \rightharpoonup 0$.
Moreover, if $(u_n) \subset \mathcal{N}$ and $u_n \rightharpoonup u_0$ then $u_0 \not \equiv 0$. Indeed, if $u_0 \equiv 0$ then, from $J_0(u_n)-J(u_n)=\Phi'(u_n)u_n=0$ and the weak continuity of $J$ we deduce that $J_0(u_n) \to 0$, so that, by (B), $u_n \to 0$, which contradicts the fact that $\mathcal{N}$ is bounded away from zero. The rest of the proof holds without further modifications. \\

\item If $\Phi$ is weakly lower semi-continuous then, instead of (4) and (5), one may require only that for every $u \in X \setminus \{0\}$ the map $\dis t \mapsto \frac{\Phi'(tu)u}{t^{p-1}}$ is decreasing in $(0,\infty)$. Note indeed that one may still obtain a minimizing sequence for $c$ such that $u_n \rightharpoonup u_0$ and $u_0 \not \equiv 0$. From the weak lower semicontinuity of $\Phi$ and the fact that $\Phi(u_n)=\max_{t>0} \Phi(tu_n)$ we deduce that $$c \leq \Phi(t_0u_0) \leq \liminf \Phi(t_0 u_n) \leq \liminf \Phi(u_n)=c,$$
i.e. $c$ is achieved.\\

\item Unlike \cite{SW}, we don't make use of the Palais-Smale condition of $\Phi$ to show that $c$ is achieved. Indeed, note that the proof of Theorem \ref{tp} does not require the strong convergence of a minimizing sequence for $c$.\\ 

\item From the proof of Theorem \ref{tp}, we shall highlight the following result: if $(u_n) \subset \mathcal{N}$ is such that $(\Phi(u_n))$ is bounded from above then $(u_n)$ is bounded.\\
\end{enumerate}
}
\end{rmk}

Following \cite{SW}, we say that $\Phi$ satisfies the Palais-Smale condition on $\mathcal{N}$ if any Palais-Smale sequence of $\Phi$ which is moreover in $\mathcal{N}$ contains a convergent subsequence.

Combining Theorem \ref{tp} above and Theorem 2 and Corollary 10 from \cite{SW}, we get the following result:

\begin{cor}\label{cp}
Under the assumptions of Theorem \ref{tp}, assume in addition that $\Phi$ is even and satisfies the Palais-Smale condition on $\mathcal{N}$. Then $\Phi$ has infinitely many pairs of critical points.
\end{cor}

\section{Applications}

We apply now Theorem \ref{tp} and Corollary \ref{cp} to the equations \eqref{quasi}, \eqref{kir} and \eqref{anis}. Let us recall that $\Omega \subset \Re^N$ ($N \geq 3$) is a bounded domain, $f:\Re \rightarrow
[0,\infty)$ is an odd $\mathcal{C}^{1}$ function and $F(t)=\displaystyle\int_{0}^{t}f(s)\ ds$,  for $t \in \Re$.

\subsection{A quasilinear equation}
We assume that $a:[0,\infty)\rightarrow
[0,\infty)$ is $\mathcal{C}^{1}$ in $(0,\infty)$ and we set $A(t)=\displaystyle\int_{0}^{t}a(s)\ ds$ for $t \in \Re$.

\begin{cor}\label{c1}
Under the above assumptions on $a$, assume in addition that there exist $p\geq q>1$ such that:
\begin{enumerate}
\item $k_0\left( 1+t^{\frac{q-p}{p}}\right) \leq a(t) \leq k_1\left( 1+t^{\frac{q-p}{p}}\right), \quad \forall t>0$, where $k_0,k_1$ are positive constants.
\item $a$ is non-increasing.
\item  $t\mapsto a(t^p)t^p$ and $t \mapsto A(t^p)- a(t^p)t^p$ are convex in  $(0,\infty)$.
\item $\dis \lim_{t \rightarrow 0} \frac{f(t)}{t^{q-1}}=0$.
\item $\dis \lim_{t \rightarrow \infty} \frac{F(t)}{t^p}=\infty $.
\item $\dis \lim_{t \rightarrow \infty} \frac{f(t)}{t^{\alpha-1}}=0$ for some  $\alpha \in (p,p^*)$.
\item $t\mapsto  \frac{f(t)}{t^{p-1}}$
 is increasing on $(0,\infty)$. 
\end{enumerate}
Then \eqref{quasi} has a nontrivial and non-negative ground state.
\end{cor}
 
\begin{proof}
First of all, note that $A(t)$ is well-defined in view of (1).
Let $X=W_0^{1,p}(\Omega)$  with $\|u\|=\left(\int_\Omega |\nabla u|^p\right)^{\frac{1}{p}}$. We set, for $u \in X$,
\begin{equation}
\label{ph1}
\Phi(u)=\frac{1}{p}\displaystyle\int_{\Omega}A(|\nabla u|^p)-\int_\Omega F(u).
\end{equation}
From (1) we infer that 
\begin{equation}
\label{ina}
k_0\left(t^p+ t^q\right)\leq a(t^p)t^p \leq k_1 \left(t^p+ t^q\right),  \quad \forall t>0
\end{equation}
and
\begin{equation}
\label{inag}
k_0\left(t^p+ \frac{p}{q} t^q\right)\leq A(t^p)\leq k_1 \left(t^p+ \frac{p}{q} t^q\right), \quad \forall t>0.
\end{equation}
On the other hand, from (4) and (6) we infer that for any $\varepsilon>0$ there exists $C_{\varepsilon}>0$ such that 
\begin{equation}
\label{ef}
|f(t)| \leq \varepsilon |t|^{q-1} + C_\varepsilon |t|^{\alpha-1}, \quad \forall t \in \Re.
\end{equation}
Since $1<q\leq p$, it follows that $\Phi$ is a $\mathcal{C}^1$ functional on $X$.
From \eqref{ef} and the continuity of the embeddings $ X \subset L^{\alpha}(\Omega)$ and $W_0^{1,q}(\Omega) \subset L^q(\Omega)$, we infer that for any $\varepsilon>0$ there exists $C_{\varepsilon}>0$ such that 
\begin{equation*}
\left|\int_\Omega f(u)u \right|\leq \varepsilon \int_\Omega |\nabla u|^q +C_\varepsilon \|u\|^{\alpha} \quad \forall u \in X.
\end{equation*}
Taking $\varepsilon>0$ sufficiently small and using \eqref{ina}, we get
$$\Phi'(u)u \geq (k_0-\varepsilon) \int_\Omega |\nabla u|^q +k_0\|u\|^p - C_\varepsilon \|u\|^{\alpha} \quad \forall u \in X,$$
and consequently
$$\liminf_{\|u\| \to 0} \frac{\Phi'(u)u}{\|u\|^p}>0.$$
From \eqref{inag}, note also that $$\Phi(u)\geq \frac{k_0}{p}\|u\|^p -\int_\Omega F(u)$$
and by the compact Sobolev embedding $X \subset L^{\alpha}(\Omega)$, the  functional $u \mapsto \int_\Omega F(u)$ is weakly continuous on $X$. Still from \eqref{inag}, we have
$$\frac{\Phi(tu)}{t^p}\leq \frac{k_1}{q}t^{q-p} \int_\Omega |\nabla u|^q+\frac{k_1}{p}\|u\|^p-\int_\Omega \frac{F(tu)}{t^p} \rightarrow -\infty,$$
uniformly for $u$ on weakly compact subsets of $X \setminus \{0\}$, by (5).
From (2) and (7) we have that
$$t \mapsto \frac{\Phi'(tu)u}{t^{p-1}}=\int_\Omega a\left(t^p|\nabla u|^p\right) |\nabla u|^p-\int_\Omega \frac{f(tu)}{t^{p-1}}u$$ is decreasing on $(0,\infty)$ for every $u \neq 0$.
Furthermore, it is clear that
$$
t \mapsto A(t^p)-a(t^p)t^p \ \ \mbox{is
non-decreasing in } (0,\infty).
$$
On the other hand, (7) provides that $$t \mapsto \frac{1}{p}f(t)t-F(t)  \ \ \mbox{is
increasing in } (0,\infty).$$
Thus $t \mapsto \Phi(tu)-\frac{1}{p}\Phi'(tu)tu$ is increasing in $(0,\infty)$ for every $u \neq 0$.

Finally, (3) yields that $u \mapsto \Phi'(u)u$ and $u \mapsto \Phi(u)-\frac{1}{p}\Phi'(u)u$ are weakly lower semi-continuous on $X$.
Therefore Theorem \ref{tp} applies with $r=p$ and since $\Phi$ is even, we infer  that $\Phi$ has a nontrivial and non-negative ground state.
\end{proof}

Let $\Psi:X \rightarrow \Re$ be a $\mathcal{C}^1$ functional.
Recall that $\Psi'$ belongs to the class $(S_+)$ condition if 
$$u_n \rightharpoonup u_0 \text{ in } X, \quad \limsup \Phi'(u_n)(u_n-u_0) \leq 0 \quad \Longrightarrow \quad u_n \rightarrow u_0 \text{ in } X.$$
Set $\Psi(u)=\int_\Omega A(|\nabla u|^p)$ for $u \in W_0^{1,p}(\Omega)$. It is known that if $t \mapsto A(t^p)$ is strictly convex and satisfies
$$ a_1t^p-b_1 \leq A(t^p)\leq a_2t^p+b_2, \quad \forall t>0$$ for some positive constants $a_1,a_2,b_1,b_2$, then $\Psi'$ belongs to the class $(S_+)$ (see \cite{DU} for a proof).

As a consequence of Remark \ref{r1}-(4), we see that under the assumptions of Corollary \ref{c1}  and the condition \begin{equation}
\label{conv} t \mapsto A(t^p) \ \text{is strictly convex in} \ (0,\infty) 
 \end{equation}
the functional $\Phi$ given by \eqref{ph1} satisfies the Palais-Smale condition on $\mathcal{N}$. We infer then the following result:
 
\begin{cor}\label{c11}
Under the assumptions of Corollary \ref{c1}, assume in addition that \eqref{conv} holds. Then the problem \eqref{quasi} has infinitely many pairs of solutions.
\end{cor}

\begin{rmk}
One may easily check that Corollaries \ref{c1} and \ref{c11} apply in particular to $a(t) \equiv 1$ and $a(t)=1+t^{\frac{q-p}{p}}$, with $p>q>1$, which correspond to the operators $-\Delta_p$ and $-\Delta_p-\Delta_q$, respectively.
\end{rmk}

\medskip

\subsection{A nonlocal equation}
Let $N=3$, so that $2^*=6$. We assume that $M:[0,\infty)\rightarrow
[0,\infty)$ is a $\mathcal{C}^{1}$ function and we set $\hat{M}(t)=\displaystyle\int_{0}^{t}M(s)\ ds$ for $t \in \Re$.

\begin{cor}\label{c2}
Under the above assumptions on $M$, assume in addition:
\begin{enumerate}
\item $M$ is increasing and $M(0):=m_0>0$.
\item $t\mapsto\displaystyle\frac{M(t)}{t}$ is decreasing.
\item $\dis \lim_{t \rightarrow 0} \frac{f(t)}{t}=0$.
\item $\dis \lim_{t \rightarrow \infty} \frac{F(t)}{t^4}=\infty $.
\item $\dis \lim_{t \rightarrow \infty} \frac{f(t)}{t^{\alpha-1}}=0$ for some  $\alpha \in (4,6)$.
\item $t\mapsto  \frac{f(t)}{t^{3}}$ is increasing.
\end{enumerate}
Then \eqref{kir} has a nontrivial and non-negative ground state.
\end{cor}

\begin{proof}
Let $X=H_0^1(\Omega)$ with $\|u\|=\left(\int_\Omega |\nabla u|^2\right)^{\frac{1}{2}}$. We set, for $u \in X$,
\begin{equation}
\label{phk}
\Phi(u)=\frac{1}{2}\hat{M}\left(\|u\|^2\right)-\int_\Omega F(u).
\end{equation}
Note that $\| \cdot \|$ is a $\mathcal{C}^1$ functional on $X$, so that $u \mapsto \hat{M}(\|u\|^2)$ is $\mathcal{C}^1$ as well.
From (3) and (5) we have that for every $\varepsilon>0$ there exists $C_{\varepsilon}>0$ such that 
\begin{equation}
\label{ef2}
|f(t)|\leq \varepsilon |t| +C_\varepsilon |t|^{\alpha-1}, \quad \forall t \in \Re.
\end{equation}
Thus $\Phi$ is a $\mathcal{C}^1$ functional.
Using (1) and \eqref{ef2}, we have $$\Phi'(u)u=M(\|u\|^2)\|u\|^2 - \int_\Omega f(u)u \geq m_0\|u\|^2- \int_\Omega f(u)u.$$
From \eqref{ef2} and the continuous embedding $X \subset L^\alpha(\Omega)$, 
we get
\begin{equation*}
\left|\int_\Omega f(u)u \right|\leq \varepsilon \|u\|^2 +C_\varepsilon \|u\|^\alpha, \quad \forall u \in X.
\end{equation*}
Thus we have, for $u \in X$,
$$\Phi'(u)u\geq m_0\|u\|^2-\varepsilon \|u\|^2 -C_\varepsilon \|u\|^\alpha.$$ 
Taking $\varepsilon>0$ sufficiently small, we get
$$\liminf_{\|u\| \to 0} \frac{\Phi'(u)u}{\|u\|^2}>0.$$
Moreover, from (2) we infer that
$$\hat{M}(t)=\int_0^t \frac{M(s)}{s} s\ ds\geq \frac{M(t)}{t}\int_0^t s\ ds=\frac{1}{2}M(t)t, \quad \forall t>0,$$
and consequently
$$ \Phi(u)\geq \frac{m_0}{2}\|u\|^2 -\int_\Omega F(u).$$
By \eqref{ef2} and the compact  embedding $X \subset L^\alpha(\Omega)$, the  functional $u \mapsto \int_\Omega F(u)$ is weakly continuous on $X$.
Now, from (2) we have $M(t)\leq M(1)t$ for $t\geq 1$, so that
$$M(t)\leq M(1)t+C, \quad \forall t\geq 0$$
for some constant $C>0$. Consequently we have, for $u \in X$,
$$\Phi(u)\leq C_1\|u\|^4+C_2\|u\|^2-\int_\Omega F(u),$$
for some $C_1,C_2>0$, so that, by (4), $\frac{\Phi(tu)}{t^4} \rightarrow -\infty$ uniformly on weakly compact subsets of $X \setminus \{0\}$. From (2) and (6), it follows that
$$t \mapsto \frac{\Phi'(tu)u}{t^3}=\frac{1}{t^2}M(t^2\|u\|^2)\|u\|^2-\int_\Omega \frac{f(tu)}{t^3}u$$ is decreasing for every $u \neq 0$. 

Furthermore, note that (2) yields $tM'(t)\leq M(t)$ for any $t>0$, and consequently
$$t \mapsto \frac{1}{2}\hat{M}(t)-\frac{1}{4}M(t)t \quad \text{ is increasing}.$$
Hence $t \mapsto \Phi(tu)-\frac{1}{4}\Phi'(tu)tu$ is increasing in $(0,\infty)$ for every $u \neq 0$.

Finally, since $M$ and  $t \mapsto \frac{1}{2}\hat{M}(t)-\frac{1}{4}M(t)t$ are increasing, the mappings $$u \mapsto M(\|u\|^2) \quad \text{and} \quad u \mapsto \frac{1}{2}\hat{M}(\|u\|^2)-\frac{1}{4}M(\|u\|^2)\|u\|^2$$
are weakly lower semicontinuous on $X$.
Therefore Theorem \ref{tp} applies with $r=2$ and $p=4$. Note also that $\Phi$ is even. The proof is now complete.
\end{proof}

\begin{cor}\label{c22}
Under the assumptions of Corollary \ref{c2}, the problem \eqref{kir} has infinitely many pairs of solutions.
\end{cor}

\begin{proof} Let $(u_n) \subset \mathcal{N}$ be a Palais-Smale sequence for $\Phi$, defined in \eqref{phk}. By Remark \ref{r1}-(4), we know that $(u_n)$ is bounded, so that, up to a subsequence, $u_n \rightharpoonup u_0$ in $X$. From \eqref{ef2}, we know that
$$\int_\Omega f(u_n)(u_n-u_0)=o(1).$$ 
Hence
$$M(\|u_n\|^2) \int_\Omega \nabla u_n \nabla (u_n-u_0)=\Phi'(u_n)(u_n-u_0)+o(1)=o(1).$$
Since $M$ is continuous and $M(t)\geq m_0>0$ for all $t\geq 0$, we infer that $M(\|u_n\|^2)$ is bounded and bounded away from zero, so that $$\int_\Omega \nabla u_n \nabla (u_n-u_0)=o(1).$$  By the uniform convexity of $H_0^1(\Omega)$, we get $u_n \rightarrow u_0$.
\end{proof}

\begin{rmk} Besides $M(t)=at+b$, with $a,b>0$,
Corollaries \ref{c2} and \ref{c22} apply also to $M(t) = m_{0} + \ln(1 + t)$ and $
M(t)=m_{0}+\displaystyle\sum_{i=1}^{k}b_{i}t^{\gamma_{i}}
$, where $b_{i}\geq 0$ and $\gamma_{i}\in (0,1]$ for $i=1,...,k$, with $b_i>0$ for at least one $i$.
\end{rmk}
\medskip

\subsection{A anisotropic equation}
Let $1<p_1\leq p_2 \leq ...\leq p_N$ be such that $\sum_{i=1}^N \frac{1}{p_i}>1$ and $p_N < p^*$, where $p^*=\frac{N}{\left(\sum_{i=1}^N \frac{1}{p_i}\right)-1}$. If $r>1$, we denote by $\|v\|_r$ the norm of $v$ in $L^r(\Omega)$.

\begin{cor}\label{c3}
Under the above assumptions, assume in addition:
\begin{enumerate}
\item $\lim_{t \to 0^+} \frac{f(t)}{t^{p_1-1}}=0$.
\item $\lim_{t \to \infty} \frac{F(t)}{t^{p_N}}=\infty$.
\item $\lim_{t \to \infty} \frac{f(t)}{t^{\alpha-1}}=\infty$ for some $\alpha \in (p_N,p^*)$.
\item $t \mapsto \frac{f(t)}{t^{p_N-1}}$ is increasing.
\end{enumerate}
Then \eqref{anis} has a nontrivial and non-negative ground state.
\end{cor}

\begin{proof}
Let
 $X=\mathcal{D}_0^{1,\overrightarrow{p}}(\Omega)$ be the completion of $\mathcal{C}_0^{\infty}(\Omega)$ with respect to the norm $\|u\|=\displaystyle \sum_{i=1}^N \|\partial_i u\|_{p_i}$. It is known that $X$ is a reflexive Banach space which embedds continuously in $L^q(\Omega)$  if $q \in [1,p^*]$, and compactly if $q\in [1,p^*)$, cf \cite{FGK}. 
 
We set, for $u \in X$,
\begin{equation}
\label{phan}
\Phi(u)=\sum_{i=1}^N \frac{1}{p_i}\int_\Omega |\partial_i u|^{p_i}-\int_\Omega F(u).
\end{equation}
Note that if $\|u\|\leq 1$ then $\|\partial_i u\|_{p_i}<1$ for $i=1,..,N$, and since $p_i\leq p_N$ we get $\|\partial_i u\|_{p_i}^{p_i} \geq \|\partial_i u\|_{p_i}^{p_N}$ for $i=1,..,N$. Thus
$$\sum_{i=1}^N \|\partial_i u\|_{p_i}^{p_i} \geq \sum_{i=1}^N \|\partial_i u\|_{p_i}^{p_N} \geq C \left(\sum_{i=1}^N \|\partial_i u\|_{p_i}\right)^{p_N} = C\|u\|^{p_N}.$$
Using (1), (3) and the continuous embedding $X \subset L^{\alpha}(\Omega)$, we have that for any $\varepsilon >0$ there exists $C_\varepsilon>0$ such that 
\begin{equation*}
\left|\int_\Omega f(u)u \right|\leq \varepsilon \|\partial _1 u\|_{p_1}^{p_1} +C_\varepsilon \|u\|^\alpha, \quad \forall u \in X.
\end{equation*}
Hence
\begin{eqnarray*}
\Phi'(u)u&=&\sum_{i=1}^N \|\partial_i u\|_{p_i}^{p_i} - \int_\Omega f(u)u \geq \sum_{i=1}^N \|\partial_i u\|_{p_i}^{p_i}  -\varepsilon \|\partial _1 u\|_{p_1}^{p_1}-C_\varepsilon \|u\|^\alpha\\
&\geq & (1-\varepsilon)\sum_{i=1}^N \|\partial_i u\|_{p_i}^{p_i}  -C_\varepsilon \|u\|^\alpha \geq C(1-\varepsilon)\|u\|^{p_N} - C_\varepsilon \|u\|^\alpha.
\end{eqnarray*}
Taking $\varepsilon<1$ we get $$\liminf_{u \to 0} \frac{\Phi'(u)u}{\|u\|^{p_N}}>0.$$
Note also that $$\Phi(u) =I_0(u)-I(u)\quad \text{and} \quad \Phi'(u)u=J_0(u)-J(u),$$
where $$I_0(u)=\sum_{i=1}^N \frac{1}{p_i}\|\partial_i u\|_{p_i}^{p_i},\quad  I(u)=\int_\Omega F(u),$$
$$J_0(u)=\sum_{i=1}^N \|\partial_i u\|_{p_i}^{p_i}\quad \text{and} \quad J(u)=\int_\Omega f(u)u.$$
From the compact embedding $X \subset L^{\alpha}(\Omega)$, it follows that $I$ and $J$ are weakly continuous. 
Moreover there exists $C>0$ such that $$I_0(tu)\geq \frac{t^{p_1}}{p_N} \sum_{i=1}^N \|\partial_i u\|_{p_i}^{p_i}  \geq C \frac{t^{p_1}}{p_N} \|u\|^{p_N}$$
if $t>1$ and $\|u\| \leq 1$. In particular, if $u \in \mathcal{S}$ then $I_0(tu)\to \infty$ as $t \to \infty$. 
In addition, it is clear that $J_0(u_n) \to 0$ if and only if $u_n \to 0$ in $X$. From (2) we have
$$\frac{\Phi(tu)}{t^{p_N}}= \sum_{i=1}^N \frac{t^{p_i-p_N}}{p_i}\|\partial_i u\|_{p_i}^{p_i} - \int_\Omega \frac{F(tu)}{t^{p_N}} \to -\infty,$$
whereas, from (4), it follows that 
$$t \mapsto t^{1-p_N}\Phi'(tu)u= \sum_{i=1}^N t^{p_i-p_N} \|\partial_i u\|_{p_i}^{p_i} - \int_\Omega \frac{f(tu)}{t^{p_N-1}}\ \ \text{is decreasing in } (0,\infty).$$
Still by (4), we have that $$t \mapsto \frac{1}{p_N}f(t)t-F(t)\ \ \text{is increasing in } (0,\infty).$$
Thus for any $u \neq 0$ $$t \mapsto \Phi(tu)-\frac{1}{p_N}\Phi'(tu)tu=\sum_{i=1}^N \left(\frac{1}{p_i}-\frac{1}{p_N} \right)t^{p_i-p_N}\| \partial_i u\|_{p_i}^{p_i}+\int_\Omega \left(\frac{1}{p_N}f(tu)t-F(tu)\right)$$ is increasing in $(0,\infty)$.

Finally, it is clear that $u \mapsto \Phi'(u)u$ and $\mapsto \Phi(u)-\frac{1}{p_N}\Phi'(u)u$ are weakly lower semicontinuous on $X$, and $\Phi$ is even.

In view of Theorem \ref{tp} and Remark \ref{r1}-(1), we infer that \eqref{anis} has a nontrivial and non-negative ground state.
\end{proof}

\begin{cor}
Under the assumptions of Corollary \ref{c3}, the problem \eqref{anis} has infinitely many pairs of solutions.
\end{cor}

\begin{proof}
Let $(u_n) \subset \mathcal{N}$ be a Palais-Smale sequence for $\Phi$, defined in \eqref{phan}. Since $(u_n)$ is bounded,  passing to a subsequence, if necessary, we have
$$
u_{n}\rightharpoonup u \ \ \mbox{in} \ \
D_{0}^{1,\overrightarrow{p}}(\Omega), \ \
u_{n}\rightarrow u \ \ \mbox{in} \ \ L^{\sigma}(\Omega) 
\mbox{ for} \ \ \sigma \in [1, p^{*}),\ \
\text{and} \ \
u_{n}\rightarrow u \ \ \mbox{a.e in} \ \ \Omega.
$$
Since $|f(t)|\leq C(1+|t|^{\alpha-1})$ for every $t \in \Re$, we get
\begin{eqnarray}\label{converg1}
\displaystyle\int_{\Omega}f(u_{n})u_n  -
\displaystyle\int_{\Omega}f(u_n)u = o(1).
\end{eqnarray}
Now, from $u_{n}\rightharpoonup u$ in $D_{0}^{1,\overrightarrow{p}}(\Omega)$ we have
\begin{eqnarray}\label{converg2}
\displaystyle\sum^{N}_{i=1}\displaystyle\int_{\Omega}|\partial_i u|^{p_{i}-2}\partial_i u\ \partial_i u_n
-\displaystyle\sum^{N}_{i=1}\displaystyle\int_{\Omega}|\partial_i u|^{p_i}= o(1).
\end{eqnarray}
We use now the standard inequality (cf. \cite{Si}) $$
(|x|^{p-2}x-|y|^{p-2}y)(x-y) \geq \begin{cases} C_{p}|x-y|^{p}, & \text{if } p\geq 2\\ C_p\frac{|x-y|^{2}}{(|x|+|y|)^{2-p}}, & \text{if } 1<p<2, \end{cases}
$$ which holds for some $C_p>0$.
Note that if $1<p_i< 2$ then 
\begin{eqnarray*}
\|\partial_i u_n- \partial_i u\|_{p_i}^{p_i} &=&\int_\Omega \frac{|\partial_i u_n- \partial_i u|^{p_i}}{\left(|\partial_i u_n|+|\partial_i u|\right)^{\frac{p_i(2-p_i)}{2}}}\left(|\partial_i u_n|+|\partial_i u|\right)^{\frac{p_i(2-p_i)}{2}} \\&\leq & \left(\int_\Omega \frac{|\partial_i u_n- \partial_i u|^2}{\left(|\partial_i u_n|+|\partial_i u|\right)^{2-p_i}}\right)^{\frac{p_i}{2}} \left(\int_\Omega \left(|\partial_i u_n|+|\partial_i u|\right)^{p_i}\right)^{\frac{2-p_i}{2}}\\ &\leq &
C \left(\int_{\Omega}\left(|\partial_i u_n|^{p_{i}-2}\partial_i u_n- |\partial_i u|^{p_{i}-2}\partial_i u\right)\left(\partial_i u_n-\partial_i u\right)\right)^{\frac{p_i}{2}}
\end{eqnarray*} 
Thus
$$\int_{\Omega}\left(|\partial_i u_n|^{p_{i}-2}\partial_i u_n- |\partial_i u|^{p_{i}-2}\partial_i u\right)\left(\partial_i u_n-\partial_i u\right)\geq 
\begin{cases}
C\|\partial_i u_n- \partial_i u\|_{p_i}^{p_i}, & \text{if } p_i\geq2 \\
C\|\partial_i u_n- \partial_i u\|_{p_i}^2, & \text{if } 1<p_i<2.
\end{cases}
$$
for some $C>0$,
and consequently
\begin{eqnarray*}
\Phi'(u_n)(u_n-u)&=& \sum_{i=1}^N \int_{\Omega}|\partial_i u_n|^{p_{i}-2}\partial_i u_n\left(\partial_i u_n-\partial_i u\right)+o(1)\\
&=&
\sum_{i=1}^N \int_{\Omega}\left(|\partial_i u_n|^{p_{i}-2}\partial_i u_n- |\partial_i u|^{p_{i}-2}\partial_i u\right)\left(\partial_i u_n-\partial_i u\right)+o(1)\\
&\geq &C\left(\sum_{p_i\geq 2} \|\partial_i u_n- \partial_i u\|_{p_i}^{p_i} + \sum_{1<p_i<2} \|\partial_i u_n- \partial_i u\|_{p_i}^2\right)+o(1)
\end{eqnarray*}
Therefore $\|\partial_i u_n- \partial_i u\|_{p_i} \rightarrow 0$ for $i=1,...,N$, so that $u_{n}\rightarrow u$ in
$D_{0}^{1,\overrightarrow{p}}(\Omega)$ and the proof is complete.
\end{proof}

\end{document}